\documentclass{amsart}

\usepackage[utf8]{inputenc}

\usepackage{amsmath, amsfonts,amssymb,amsthm,array}
\usepackage{graphicx}
\usepackage{stackengine}
\usepackage{thmtools,thm-restate}
\usepackage{enumitem}
\usepackage{float}
\usepackage{comment}
\usepackage{xifthen}
\usepackage{latexsym}
\usepackage{stmaryrd}
\usepackage{enumitem}
\usepackage[bookmarks,hyperfootnotes=false]{hyperref}
\usepackage{multirow}
\usepackage{xcolor}
\usepackage{caption}
\colorlet{darkishRed}{red!80!black}
\colorlet{darkishBlue}{blue!60!black}
\colorlet{darkishGreen}{green!60!black}
\usepackage{hyperref}
\hypersetup{
pdftitle = {Ends of Digraphs},
pdfsubject = {Ends of Digraphs},
pdfauthor = {Carl Bürger and Ruben Melcher},
pdfkeywords = {},
draft = false,
bookmarksopen=true}

\newtheorem{innercustomthm}{Theorem}

\newenvironment{customthm}[1]
  {\renewcommand\theinnercustomthm{#1}\innercustomthm}
  {\endinnercustomthm}

\newtheorem{theorem}{Theorem}[section]
\newtheorem{mainresult}{Theorem}

\newtheorem{proposition}[theorem]{Proposition}
\newtheorem{cor}[theorem]{Corollary}
\newtheorem{lemma}[theorem]{Lemma}

\theoremstyle{definition}

\setenumerate{label={\normalfont (\roman*)},itemsep=0pt}

\colorlet{darkishGreen}{green!60!black}

\newcommand{\cO}{\mathcal{O}}
\newcommand{\cB}{\mathcal{B}}
\newcommand{\cX}{\mathcal{X}}

\newcommand{\cU}{\mathcal{U}}

\newcommand{\cP}{\mathcal{P}}

\newcommand{\dmtop}{\textsc{DTop}}

\newcommand{\N}{\mathbb{N}}

\newcommand{\nao}[1][]{%
\ifthenelse{\equal{#1}{}}{\trianglelefteq_T}{\trianglelefteq_{T_{#1}}\!}%
}

\newcommand{\dS}{\mathstrut\mkern2.5mu S\mkern-13mu\raise1.4ex%
\hbox{$\leftrightarrow$}}

\def\lowfwd #1#2#3{{\mathop{\kern0pt #1}\limits^{\kern#2pt\raise.#3ex
\vbox to 0pt{\hbox{$\scriptscriptstyle\rightarrow$}\vss}}}}
\def\lowbkwd #1#2#3{{\mathop{\kern0pt #1}\limits^{\kern#2pt\raise.#3ex
\vbox to 0pt{\hbox{$\scriptscriptstyle\leftarrow$}\vss}}}}

\def\ve{\kern-1.5pt\lowfwd e{1.5}2\kern-1pt}
\def\ev{\kern-1pt\lowbkwd e{0.5}2\kern-1pt}
\def\vf{\kern-2pt\lowfwd f{2.5}2\kern-1pt}

\stackMath

\newenvironment{myindentpar}[1]%
{\begin{list}{}%
          {\setlength{\leftmargin}{#1}}%
          \item[]%
}
{\end{list}}

\graphicspath{{Main/}} 

\begin{document}

\title[Ends of digraphs]{Ends of digraphs II: the topological point of view} 

\author{Carl B\"{u}rger}
\author{Ruben Melcher}
\address{University of Hamburg, Department of Mathematics, Bundesstraße 55 (Geomatikum), 20146 Hamburg, Germany}
\email{carl.buerger@uni-hamburg.de, ruben.melcher@uni-hamburg.de}

\keywords{infinite digraph; end; limit edge; topology; compact; solid; inverse limit; Eulerian; strongly connected}%
\subjclass[2010]{05C20, 05C45, 05C63}

\maketitle

\begin{abstract} In a series of three papers we develop an end space theory for digraphs. Here in the second paper we introduce the topological space $|D|$ formed by a digraph $D$ together with its ends and limit edges. We then characterise those digraphs that are compactified by this space. Furthermore, we show that if $|D|$ is compact, it is the inverse limit of finite contraction minors of $D$.  

To illustrate the use of this we extend  to the space $|D|$ two statements about finite digraphs that do not generalise verbatim to infinite digraphs.  The first statement is the characterisation of finite Eulerian digraphs by the condition that the in-degree of every vertex equals its out-degree. The second statement is the characterisation of strongly connected finite digraphs by the existence of a closed Hamilton walk.
\end{abstract}

\section{Introduction} 
\noindent Ends of graphs are one of the most important concepts for the study of infinite graphs.
In a series of three papers we develop an end space theory for digraphs.
See~\cite{EndsOfDigraphsI} for a comprehensive introduction to the entire series of three papers  (\cite{EndsOfDigraphsI}, \cite{EndsOfDigraphsIII} and this paper) and a brief overview of all our results.

In 2004, Diestel and Kühn~\cite{CyclesI} introduced a topological framework for infinite graphs which makes it possible to extend theorems about finite graphs to infinite graphs that do not generalise verbatim. The main point is to consider not only  the graph itself but the graph together with its ends, and to equip both together with a suitable topology. For locally finite graphs $G$, this space $|G|$ coincides with the Freudenthal compactification of $G$~\cite{DiestelBook5, diestel2011locallyI}.

Diestel and Kühn's approach has become standard and has lead to several results found by various authors. Examples include Nash-William's tree-packing theorem~\cite{DiestelBook3}, Fleischner's Hamiltonicity theorem \cite{AgelosFleischner}, and Whitney's planarity criterion~\cite{duality}. In the formulation of these theorems, topological arcs and circles take the role of paths and cycles, respectively.

To illustrate this, consider Euler's theorem that a connected finite graph contains an Euler tour if and only if every vertex has even degree. This statement fails for infinite graphs, since a closed walk in a connected infinite graph cannot traverse all its infinitely many edges. Diestel and K\"uhn~\cite{CyclesI} extended Euler's Theorem  to the space $|G|$ for locally finite graphs $G$, as follows.   
A  \emph{topological Euler tour} of $ \vert G \vert $ is a continuous map $\sigma \colon S^1 \to | G | $ such that every inner point of an edge of $G$ is the image of exactly one point of~$S^1$. Hence a topological Euler tour `traverses' every edge exactly once. 
Diestel and K\"uhn showed that for a connected locally finite graph $G$ the space $|G|$ admits a topological Euler tour if and only if every finite cut in $G$ is even. Note that in a finite graph every vertex has even degree if and only if all finite cuts are even, but even for locally finite infinite graphs the latter statement is stronger. Theorem~\ref{thm: eulerian characterisation} below is a directed analogue of the Diestel-Kühn theorem about topological Euler tours.

In the first paper of this series~\cite{EndsOfDigraphsI} we introduced the concept of \emph{ends} and \emph{limit edges} of a digraph. 
A \emph{directed ray} is an infinite directed path that has a first vertex (but no last vertex). The directed subrays of a directed ray are its \emph{tails}. For the sake of readability we shall omit the word `directed' in  `directed path' and `directed ray' if there is no danger of confusion. We call a ray in a digraph \emph{solid} in $D$ if it has a tail in some strong component of $D-X$ for every finite vertex set $X\subseteq V(D)$. We call two solid rays in a digraph $D$ \emph{equivalent} if for every finite vertex set $X\subseteq V(D)$ they have a tail in the same strong component of $D-X$. The classes of this equivalence relation are the \emph{ends} of $D$. 
The ends of a digraph can be thought of as points at infinity to which its solid rays converge.

For limit edges the situation is similar. Informally, they are additional edges that naturally arise between distinct ends of a digraph, as follows. Unlike graphs, digraphs may have two  rays $R$ and $R'$ that represent distinct ends and yet there may be a set $E$ of infinitely many independent edges from $R$ to~$R'$. In this situation there will be a limit edge from the end represented by $R$ to the end represented by~$R'$, and the edges in $E$ can be thought of as converging towards this limit edge. We will recall the precise definition of limit edges in Section~\ref{section: tools and terminology}.

We begin this paper by introducing a topology, which we call \dmtop, on the space $|D|$ formed by the digraph $D$ together with its ends and limit edges. In this topology rays and edges will converge to ends and limit edges, respectively. 
As our first main result we characterise those digraphs for which $|D|$ with \dmtop\ compactifies $D$. 

For graphs $G$, a necessary condition for $|G|$ to be compact is that $G-X$ has only finitely many components for every finite vertex set $X\subseteq V(G)$: if $G-X$ has infinitely many components, then these components together with all ends living in them  will form a disjoint family of open sets, and combining this family with a suitable cover of the finite graph $G[X]$ yields an open cover of $|G|$ that has no finite subcover. In \cite{VTopComp}, Diestel proves that this necessary condition is also sufficient. In analogy to this, let us call a digraph $D$ \emph{solid} if $D-X$ has only finitely many strong components for every finite vertex set $X\subseteq V(D)$. 

Our first main result is that Diestel's characterisation carries over to digraphs:

\begin{customthm}{\ref{thm: compact iff solid}} 
The space $\vert D \vert$ is compact if and only if $D$ is solid.
\end{customthm}
\noindent We remark that for every digraph $D$ the space $|D|$ is Hausdorff and that $D$ is dense in $|D|$. Hence if $D$ is solid, the space $|D|$ is a Hausdorff compactification of $D$.

A common way to generalise statements about finite graphs to infinite graphs is to use so-called compactness arguments. These can be phrased in terms of inverse limits. For example, the ray given by K\H{o}nig's infinity lemma~\cite[Lemma~8.1.2]{DiestelBook5} exists because the inverse limit of compact discrete spaces is non-empty. In order to make this technique applicable for the space $|D|$ we provide the following:

\begin{customthm}{\ref{thm: |D| homeo inverse limit]}}
For a solid digraph $D$ the space $|D|$ is the inverse limit of finite contraction minors of $D$.
\end{customthm}
\noindent  For the precise statement of this theorem see Section~\ref{section: inverse limit}. 

Recall that our motivation for introducing a  topology on a digraph $D$, together with its ends and limit edges, was to extend to the space $|D|$ theorems about finite digraphs that would be either false, or trivial, or undefined for $D$ itself. 
As a proof of concept, we prove two such applications for $|D|$.

For our first application recall that a finite digraph with  a  connected  underlying graph contains an Euler tour if and only if the in-degree equals the out-degree at every vertex~\cite{bang2008digraphs}. This statement fails for infinite digraphs, since a closed walk can only traverse finitely many edges.

As in the case of undirected graphs, however, there is a natural topological notion of Euler tours of $|D|$. Call a continuous map $\alpha \colon [0,1] \to |D|$ that respects the direction of the edges of $|D|$  a  \emph{topological path in $|D|$}, which is \emph{closed} if $\alpha(0)=\alpha(1)$. Call a closed topological path  an \emph{Euler tour} if it traverses every edge exactly once, and call $|D|$ \emph{Eulerian} if it admits an Euler tour. See Section~\ref{section: Applications} for precise definitions.

If $|D|$ is Eulerian, then the in-degree equals the out-degree at every vertex. The converse of this fails in general. For example,  the digraph $D$ on $\mathbb{Z}$ with edges $n(n+1)$ for every $n \in \mathbb{Z}$ has in- and out-degrees $1$ at every vertex, but $|D|$ has no Euler tour.

A \emph{cut} of a digraph $D$ is an ordered pair $(V_1,V_2)$ of non-empty sets $V_1,V_2 \subseteq V(D)$ such that $V_1 \cup V_ 2 = V(D)$ and $V_1 \cap V_2 = \emptyset$. The sets $V_1$ and $V_2$ are the \emph{sides} of the cut, and its \emph{size} is the cardinality of the set of edges from $V_1$ to $V_2$. We call a cut $(V_1,V_2)$ \emph{balanced} if its size equals that of $(V_2,V_1)$. Note that in a finite digraph the in-degree at every vertex equals the out-degree if and only if all finite cuts are balanced, but our $\mathbb{Z}$-example shows that for infinite digraphs, even locally finite ones, the latter statement is stronger.

Any unbalanced finite cut in a digraph $D$ is an obstruction that prevents $|D|$ from being Eulerian: by the directed jumping arc lemma (Lemma~\ref{lemma: jump lemma}), any Euler tour enters a side of a finite cut as often as it leaves it.  A second obstruction is a vertex of infinite in- or out-degree, as an Euler tour that traverses a vertex infinitely often forces the tour to converge to that vertex (see the proof of Theorem~\ref{thm: eulerian characterisation} for details). 

As our first application we show that there are no further obstructions. A digraph is \emph{locally finite} if all of its vertices have finite in- and out-degree.  

\begin{customthm}{\ref{thm: eulerian characterisation}}
For a digraph $D$ with a connected underlying graph the following assertions are~equivalent:
\begin{enumerate}
    \item $|D|$ is Eulerian;
    \item $D$ is locally finite and every finite cut of $D$ is balanced.
\end{enumerate}
\end{customthm}

In our second application we characterise the digraphs that are strongly connected. It is easy to see that a finite digraph is strongly connected if and only if it contains a closed directed walk that contains all its vertices. We obtain the following characterisation of strongly connected infinite digraphs:

\begin{customthm}{\ref{thm: char strong digraphs}}
For a countable solid digraph $D$ the following assertions  are equivalent:
\begin{enumerate}
    \item $D$ is strongly connected;
    \item there is a closed topological path in $|D|$ that contains all the vertices of $D$.
\end{enumerate}
\end{customthm}

\noindent We remark that the requirements `countable' and `solid' for $D$ are necessary. Indeed, any closed topological path in $|D|$ that traverses uncountably many vertices also traverses uncountably many edges of $D$. This gives rise to uncountably many disjoint open intervals in $[0,1]$, which is impossible. 
Furthermore, recall that the image of a compact space under a continuous map is compact. In particular the image of every topological path that contains all the vertices of $D$ is compact in~$|D|$. Hence the closure of $V(D)$ is compact, which implies that $D$ is solid. Indeed, if $D-X$ has infinitely many strong components for some finite vertex set $X\subseteq V(D)$, then these strong components together with all the ends that live in them will form a disjoint family of open sets, and combining this family with a suitable cover of $X$ yields an open cover of the closure of $V(D)$ that has no finite subcover.

This paper is organised as follows. In Section~\ref{section: tools and terminology} we collect together the results that we need from the first paper of this series \cite{EndsOfDigraphsI}, or from general topology.  In Section~\ref{section: a topology for digraphs} we formally define the topological space $|D|$ for a given digraph $D$ and prove Theorem~\ref{thm: compact iff solid}. In Section~\ref{section: inverse limit} we define an inverse system for a given digraph~$D$ and show that the inverse limit of this system coincides with $|D|$ if $D$ is solid (Theorem~\ref{thm: |D| homeo inverse limit]}).  Finally, in Section~\ref{section: Applications} we prove our two applications of our framework, Theorem~\ref{thm: eulerian characterisation} and Theorem~\ref{thm: char strong digraphs}.

\section{Tools and terminology}\label{section: tools and terminology} 
\noindent In this section we provide the tools and terminology that we use throughout this paper. Any graph-theoretic notation not explained here can be found in Diestel's textbook \cite{DiestelBook5}. For the sake of readability, we sometimes omit curly brackets of singletons, i.e., we write $x$ instead of $\{x\}$ for a set~$x$.
Furthermore, we omit the word `directed'---for example in `directed path'---if there is no danger of confusion. 

First, we list some tools and terminology that we adopt from the first paper~\cite{EndsOfDigraphsI} of our series. If not stated otherwise we consider digraphs without parallel edges and without loops.  For a digraph $D$ we write $V(D)$ for the vertex set of $D$, we write $E(D)$ for the edge set of $D$ and $\cX=\cX(D)$ for the set of finite vertex sets of $D$. 
We write edges as order pairs $(v,w)$ of vertices $v,w\in V(D)$, and usually we write $(v,w)$ simply as $vw$; except if $D$ is a multi-digraph in which case we write edges of $D$ as triples $(e,v,w)$. The vertex $v$ is the \emph{tail} of $vw$ and the vertex $w$ its~\emph{head}.

Given sets $A, B\subseteq V(D)$ of vertices a \emph{path from $A$ to $B$}, or \emph{$A$--$B$} path is a path that meets $A$ precisely in its first vertex and $B$ precisely in its last vertex. We say that a vertex $v$ can \emph{reach} a vertex $w$ in $D$ and $w$ can be \emph{reached} from $v$ in $D$ if there is a $v$--$w$ paths in $D$. A set $W$ of vertices is \emph{strongly connected} in $D$ if every vertex of $W$ can reach every other vertex of $W$ in $D[W]$.  A vertex set $Y \subseteq V(D)$ \emph{separates} $A$ and $B$ in $D$ with $A,B\subseteq V(D)$ if every $A$--$B$ path meets $Y$, or if every $B$--$A$ path meets $Y$.

The \emph{reverse} of an edge $vw$ is the edge $wv$. More generally, the \emph{reverse} of a digraph $D$ is the digraph on $V(D)$ where we replace every edge of $D$ by its reverse, i.e., the reverse of $D$ has the edge set $\{ \, vw \mid wv \in E(D)\, \}$.
A \emph{symmetric ray} is a digraph obtained from an undirected ray by replacing each of its edges by  its two orientations as separate directed edges. Hence the reverse of a symmetric ray is a symmetric ray.

A \emph{directed ray} is an infinite directed path that has a first vertex (but no last vertex). 
The directed subrays of a ray are its \emph{tails}. Call a ray \emph{solid} in $D$ if it has a tail in some strong component of $D-X$ for every finite vertex set $X\subseteq V(D)$.
Two solid rays in $D$ are  \emph{equivalent}, if they have a tail in the same strong component of $D-X$ for every finite vertex set $X \subseteq V(D)$. We call the equivalence classes of this relation the \emph{ends} of $D$ and we write $\Omega(D)$ for the set of ends of $D$.

Similarly, the reverse subrays of a reverse ray are its \emph{tails}. We call a reverse ray \emph{solid} in $D$ if it has a tail in some strong component of $D-X$ for every finite vertex set $X\subseteq V(D)$. With a slight abuse of notation, we say that a reverse ray $R$ \emph{represents} an end $\omega$ if there is a solid ray $R'$ in $D$ that represents $\omega$ such that $R$ and $R'$ have a tail in the same strong component of $D-X$ for every finite vertex set $X\subseteq V(D)$.

For a finite set $X \subseteq V(D)$ and a strong component $C$ of $D-X$ we say that an end $\omega$ \emph{lives in} $C$ if one (equivalent every)  ray  that represents $\omega$ has a tail in $C$. We write $C(X,\omega)$ for the strong component of $D-X$ in which $\omega$ lives. For two ends $\omega$ and $\eta$ of $D$, we say that a finite vertex set $X \subseteq V(D)$ \emph{separates} $\omega$ and $\eta$ if $C(X, \omega) \neq C(X, \eta )$, i.e., if $\omega$ and $\eta$ live in distinct strong components  of $D-X$.

For vertex sets $A, B\subseteq V(D)$ let $E(A,B)$ be the set of edges from $A$ to $B$, i.e., $E(A,B)= (A\times B)\cap E(D)$. Given a subdigraph $H\subseteq D$, a \emph{bundle} of $H$ is a non-empty edge set of the form $E(C,C')$, $E(v,C)$, or $E(C,v)$ for strong components $C$ and $C'$ of $H$ and a vertex $v\in V(D) \setminus V(H)$.
We say that $E(C,C')$ is a bundle \emph{between strong components}, and $E(v,C)$ and $E(C,v)$ are bundles \emph{between a vertex and a strong component.} In this paper we consider only bundles of subdigraphs $H$ with $H=D-X$ for some $X\in \cX(D)$.

Now, consider a vertex $v\in V(D)$, two ends $\omega, \eta\in \Omega(D)$ and a finite vertex set $X\subseteq V(D)$. If $X$ separates $\omega$ and $\eta$ we write $E(X,\omega\eta)$ as short for $E(C(X,\omega),C(X,\eta))$. Similarly, if $v \in C'$ for a strong component $C' \neq C(X,\omega)$ of $D-X$ we write $E(X,v\omega)$ and $E(X,\omega v)$ as short for the edge set $E(C', C(X,\omega))$ and $E(C(X,\omega),C')$, respectively. If $v \in X$  we write $E(X,v\omega)$ and $E(X,\omega v)$ as short for $E(v, C(X,\omega))$ and $E(C(X,\omega),v)$, respectively.  Note that $E(X,\omega\eta)$, $E(X,v\omega)$ and $E(X,\omega v)$ each define a unique bundle if they are non-empty.

A direction of a digraph $D$ is a map $f$ with domain $\cX(D)$ that sends every $X \in \cX(D)$ to a strong component or a bundle of $D-X$ so that $f(X)\supseteq f(Y)$ whenever $X\subseteq Y$.\footnote{Here, as later in this context, we do not distinguish rigorously between a strong component and its set of edges. Thus if $Y$ separates $\omega$ and $\eta$ but $X\subseteq Y$ does not, the expression $f_{\omega\eta}(X)\supseteq f_{\omega\eta}(Y)$ means that the strong component $f_{\omega \eta}(X)$ of $D-X$ contains all the edges from the edge set $f_{\omega\eta}(Y)$.} We call a direction $f$ on $D$ a \emph{vertex-direction} if $f(X)$ is a strong component of $D-X$ for every $X \in \cX(D)$, and we call it an \emph{edge-direction} otherwise, i.e., if $f(X)$ is a bundle of $D-X$ for some $X \in \cX(D)$. Every end $\omega$ of  a digraph $D$ defines a direction $f_\omega$ on $D$ in that it maps $X \in \cX(D)$ to $C(X,\omega)$.
The ends of~$D$ correspond bijectively to its vertex-directions:

\begin{theorem}[{\cite[Theorem~2]{EndsOfDigraphsI}}]\label{thm: ends correspond bijectively to vertex directions}
Let $D$ be any infinite digraph. The map $\omega \mapsto f_\omega$ with domain $\Omega(D)$ is a bijection between the ends and the vertex-directions of $D$.
\end{theorem}

For two distinct ends $\omega$ and  $\eta$ of a digraph $D$, we call the pair  $(\omega,\eta)$ a \emph{limit edge} \emph{from} $\omega$ \emph{to} $\eta$, if $D$ has an edge from $C(X,\omega)$ to $C(X,\eta)$ for every finite vertex set $X\subseteq V(D)$ that separates $\omega$ and $\eta$. 

For a vertex $v\in V(D)$ and an end $\omega\in \Omega(D)$ call the pair $(v,\omega)$ a \emph{limit edge} \emph{from} $v$ \emph{to} $\omega$ if $D$ has an edge from $v$ to $C(X,\omega)$ for every finite vertex set $X\subseteq V(D)$ with $v \not\in  C(X, \omega)$. Similarly, we call the pair $(\omega,v)$ a \emph{limit edge} \emph{from} $\omega$ \emph{to} $v$ if $D$ has an edge from $C(X,\omega)$ to $v$ for every finite vertex set $X\subseteq V(D)$ with $v \notin C(X, \omega)$. We  denote by $\Lambda(D)$ the set of all the limit edges of $D$ and we use the usual definitions for edges accordingly; for example we will speak of the head and the tail of a limit edge.  Every limit edge $\lambda$ defines an edge-direction as follows. We say that a limit edge $\lambda=\omega\eta$ \emph{lives} in the bundle defined by $E(X, \lambda)$ if $X \in \cX(D)$ separates $\omega$ and~$\eta$.  If $X \in \cX(D)$ does not separate $\omega$ and $\eta$, we say that $\lambda=\omega\eta$ \emph{lives} in the strong component $C(X, \omega)= C(X, \eta)$ of $D-X$.
We use similar notations for limit edges of the form $\lambda= v \omega$ or $\lambda=  \omega v$ with $v\in V(D)$ and $\omega\in \Omega(D)$: We say that the limit edge $\lambda$  \emph{lives in} the bundle $E(X,\lambda )$ if $x \not\in C(X,\omega)$ and we say that $\lambda$  \emph{lives in} the strong component $C(X, \omega)$ of $D-X$, if $v \in C(X, \omega)$. The edge-direction $f_\lambda$ defined by $\lambda$ is the edge-direction that maps every finite vertex set $X\in \cX(D)$ to the bundle or strong component of $D-X$ in which $\lambda$ lives. The limit edges of any digraph correspond bijectively to its edge-directions:

\begin{theorem}[{\cite[Theorem~3]{EndsOfDigraphsI}}]\label{thm: limit edges correspond bijectively to edge directions}
Let $D$ be any infinite digraph. The map $\lambda \mapsto f_{\lambda}$ with domain $\Lambda(D)$ is a bijection between the  limit edges and the  edge-directions of~$D$.
\end{theorem} 

Let $H$ be any fixed digraph. A \emph{subdivision} of $H$ is any digraph that is obtained from $H$ by replacing every edge $vw$ of $H$ by a path $P_{vw}$ with first vertex $v$ and last vertex $w$ so that the paths $P_{vw}$ are  internally disjoint and do not meet~$V(H)\setminus \{v,w\}$. We call the paths $P_{vw}$ \emph{subdividing} paths. If $D$ is a subdivision of $H$, then the original vertices of $H$ are the \emph{branch vertices} of $D$ and the new vertices its \emph{subdividing vertices}. 

\emph{An inflated} $H$ is any digraph that arises from a subdivision $H'$ of $H$ as follows. Replace every branch vertex $v$ of $H'$ by a strongly connected digraph $H_v$ so that the $H_v$ are disjoint and do not meet any subdividing vertex; here replacing means that we first delete $v$ from $H'$ and then add $V(H_v)$ to the vertex set and $E(H_v)$ to the edge set. Then replace every subdividing path $P_{vw}$ that starts in $v$ and ends in $w$ by an $H_v$--$H_w$ path that coincides with $P_{vw}$ on inner vertices. We call the vertex sets $V(H_v)$ the \emph{branch sets} of the inflated $H$. A \emph{necklace} is an inflated symmetric ray with finite branch sets; the branch sets of a necklace are its \emph{beads}.  Note that if a digraph $D$ contains a necklace $N$, then every ray  in $N$ is solid in $D$ and any two such rays are equivalent. Hence all rays in $N$ represent a unique end of $D$. With a slight abuse of notation, we say that a necklace $N\subseteq D$ \emph{represents} an end $\omega$ of $D$ if one (equivalently every) ray in $N$ represents~$\omega$.  For limit edges we have the following:

\begin{proposition}[{\cite[Proposition~5.1]{EndsOfDigraphsI}}]\label{lemma: char limit edge type I}
For a digraph $D$ and two distinct ends $\omega$ and $\eta$ of $D$ the following assertions are equivalent:
\begin{enumerate}
    \item $D$ has a limit edge from $\omega$ to $\eta$;
    \item there are necklaces $N_\omega \subseteq D $ and $N_\eta \subseteq D$ that represent $\omega$ and $\eta$ respectively such that every bead of $N_\omega$ sends an edge to a bead of $N_\eta$. 
\end{enumerate}
Moreover, the necklaces may be chosen disjoint from each other and such that the $n$th bead of $N_\omega$ sends an edge to the $n$th bead of $N_\eta$.
\end{proposition}

\begin{proposition}[{\cite[Proposition~5.2]{EndsOfDigraphsI}}]\label{lemma: char limit edge type II}
For a digraph $D$, a vertex $v$ and an end $\omega$ of $D$ the following assertions are equivalent:
\begin{enumerate}
    \item{$D$ has a limit edge from $v$ to $\omega$ (from $\omega$ to $v$);}
    \item{there is a necklace $N \subseteq D$ that represents $\omega$ such that $v$ sends (receives) an edge to (from) every bead of $N$.}
\end{enumerate}
\end{proposition}

For a digraph $D$ and a set $\cU$ we say that a necklace $N\subseteq D$ is \emph{attached to} $\cU$ if infinitely many beads of $N$ meet every set of $\cU$. In the first paper of this series we introduced an ordinal rank function that can be used to find out whether a digraph $D$ contains for a given set $\cU$ a necklace attached to $\cU$.
For this, consider a finite set~$\cU$ and think of $\cU$ as consisting of infinite sets of vertices. We define in a transfinite recursion the class of digraphs that have a \emph{$\cU$-rank}. A digraph $D$ has \emph{$\cU$-rank} $0$ if there is a set $U\in\cU$ such that $U\cap V(D)$ is finite. It has \emph{$\cU$-rank} $\alpha$ if it has no $\cU$-rank $<\alpha$ and there is some $X\in\cX(D)$ such that every strong component of $D-X$ has a $\cU$-rank $<\alpha$.  In the case  $U=V(D)$ we call the $U$-rank of $D$  the \emph{rank of $D$} (provided that $D$ has a $U$-rank). Note that more generally if $U \supseteq V(D)$ for a digraph, then its $U$-rank equals its rank.

\begin{lemma}[Necklace Lemma {\cite[Lemma~1]{EndsOfDigraphsI}}]\label{lemma: necklace lemma}
Let $D$ be any digraph and  $\cU$  a finite set of vertex sets of $D$. Then exactly one of the statements is true: \begin{enumerate}
    \item $D$ has a necklace attached to $\cU$;    
    \item $D$ has a $\cU$-rank. 
\end{enumerate}
\end{lemma}

An \emph{arborescence} is a rooted oriented tree $T$ that contains for every vertex \mbox{$v \in V(T)$} a directed path from the root to $v$. A \emph{directed star} is an arborescence whose underlying tree is an undirected star that is centred in the root of the arborescence. A \emph{directed comb} is the union of a ray with infinitely many finite disjoint paths (possibly trivial) that have precisely their first vertex on $R$. Hence the underlying graph of a directed comb is an undirected comb. The \emph{teeth} of a directed comb or reverse directed comb are the teeth of the underlying comb. The ray from the definition of a comb is the \emph{spine} of the comb. Given a set $U$ of vertices in a digraph, a comb \emph{attached} to $U$ is a comb with all its teeth in $U$ and a star \emph{attached} to $U$ is a subdivided infinite star with all its leaves in $U$.
The set of teeth is the \emph{attachment set} of the comb and the set of leaves is the \emph{attachment set} of the star.
We adapt the notions of `attached to' and `attachment sets' to reverse combs or reverse stars, respectively. We need the following lemma from the first paper of this \mbox{series~\cite[Lemma~3.2]{EndsOfDigraphsI}.}

\begin{lemma}[Directed Star-Comb Lemma]
Let $D$ be any strongly connected digraph and let $U\subseteq V(D)$ be infinite. Then $D$ contains a star or comb attached to $U$ and a reverse star or reverse comb attached to $U$ sharing their attachment sets.
\end{lemma}


In the second part of this section, we list the tools and terminology about inverse limits that we need. Here we follow the textbook of Zalesskii and Ribes~\cite{RibesZalesskii}. 

Let $(I,\le)$ be a \emph{directed} partially ordered set, i.e., $I$ is partially ordered by $\le$ and for any two elements  $i,j\in I$ there exist an element $k \in I$ such that $ i,j \leq k$. A collection  $\{\, X_i \mid i\in I \, \}$ of topological spaces together with continuous maps  $f_{ji} \colon X_j \to X_i$, for all $i \leq j$, is called \emph{inverse system} if $f_{ki}=f_{ji}\circ f_{kj}$ whenever $i\le j\le k$ and $f_{ii}$ is the identity on $X_i$, for all $i \in I$. We denote such an inverse system by $\{X_i, f_{ij}, I \} $. The continuous maps $f_{ji} \colon X_j \to X_i$ are called \emph{bonding maps}. The \emph{inverse limit} $\varprojlim ( X_i )_{ i \in I }$ is the subspace of the product space $\prod_{i \in  I} X_i$ that consists of all the $(x_i)_{i\in I}$ with  $f_{ji}(x_j)=x_i$ for all $i\le j$. In this setup, we write $f_i$ for the projection from $\varprojlim ( X_i )_{ i \in I }$ to~$X_i$. If all the $X_i$ are Hausdorff, the inverse limit is closed in the product space. Therefore, by Tychonoff's theorem, if all the $X_i$ are in addition compact, then the inverse limit is compact. For a topological space $Y$ together with continuous maps $ \varphi_i \colon Y \to X_i$, for all $i \in I$, the collection of maps $\{ \, \varphi_i \mid i \in I\, \}$  is called \emph{compatible} if $\varphi_i= f_{ji}  \circ \varphi_j$ for all $i\le j$. The inverse limit of an inverse system is (up to unique homeomorphism) characterised by the following \emph{universal property}:

\vspace{0,2cm}
\begin{myindentpar}{\parindent}
   \emph{For every topological space $Y$ together with  compatible maps  $\varphi_i\colon Y\to X_i$, for $i \in I$, there is a unique continuous map $\Phi \colon Y \to \varprojlim ( X_i )_{ i \in I } $ with $\varphi_i = f_i \circ \Phi$ for all $i \in I$.} 
\end{myindentpar}
\vspace{0,2cm}

In this situation, we say that the map $\Phi$ is \emph{induced} by the maps $\varphi_i$. For a topological space $Y$ together compatible maps  $\varphi_i \colon Y \to X_i $, for $i \in I$, the collection of maps $\{ \, \varphi_i \mid i \in I\, \}$ is called \emph{eventually injective} if for every two distinct $y,y' \in Y$ there is some $i \in I$ with $\varphi_i (y  ) \neq \varphi_i (y')$, see \cite[Lemma~8.8.4]{DiestelBook5}. 

\begin{lemma}[Lifting Lemma]\label{lemma: eventually injectiv} 
Let $\{X_i, f_{ij}, I\}$ be any inverse system and let $Y$ be a topological space together with eventually injective compatible maps  
$  \varphi_i\colon Y\to X_i$, for $i \in I$. Then the unique continuous map $\Phi \colon Y \to \varprojlim ( X_i )_{ i \in I } $ given by the universal property of the inverse limit is injective.
\end{lemma}

We need the following two results  \cite[Lemma~1.1.7]{RibesZalesskii} and \cite[Lemma~1.1.9]{RibesZalesskii}:

\begin{lemma}\label{lemma: image dense in inverse limit} Let $\{ X_i, f_{ij}, I\}$ be an inverse system of topological spaces over a directed set $I$, and let $ \varphi_i \colon X \to X_i $ be compatible surjections from the space $X$ onto the spaces $X_i$ $(i \in I)$. Then either $\varprojlim (X_i)_{i \in I} = \emptyset$ or the induced mapping $\Phi \colon X \to \varprojlim (X_i)_{i \in I}$ maps $X$ onto a dense subset of $\varprojlim (X_i)_{i \in I}$.
\end{lemma}

For a partially ordered set $I$ a subset $I' \subseteq I$ is called \emph{cofinal} if for all $i \in I$ there is an $i' \in I'$ with $i \leq i'$.

\begin{lemma}\label{lemma: cofinal limit homeo}
Let $\{X_i,f_{ij},I \}$ be an inverse system of compact topological spaces over a directed poset $I$ and assume that $I'$ is a cofinal subset of $I$. Then $$ \varprojlim (X_i)_{i \in I}  \cong \varprojlim (X_{i'})_{i' \in I' }.$$
\end{lemma}

Finally, we need the following theorem~\cite[Theorem~3.1.13]{EngelkingBook} from basic topology:

\begin{lemma}\label{lemma: compact into haussdorff gives hemeo}
Every continuous injective mapping of a compact space onto a Hausdorff space is a homeomorphism.
\end{lemma}

\section{A topology for digraphs}\label{section: a topology for digraphs}
\noindent In this section we define a topology on the space $\vert D \vert$ formed by a digraph $D$ together with its ends and limit edges that we call \dmtop. 

In this topological space, topological arcs and circles take the role of paths and cycles, respectively. This makes it possible to extend to the space $\vert D \vert$ statements about finite digraphs.  As an important cornerstone we  characterise, in this section, those digraphs $D$ for which $\vert D \vert$ is compact, see Theorem~\ref{thm: compact iff solid}.

Consider a digraph $D=(V,E)$ with its set $\Omega=\Omega(D)$ of  ends and its set $\Lambda=\Lambda(D)$ of limit edges. The ground set $|D|$ of our topological space is defined as follows. Take $V \cup \Omega$ together with a copy $[0,1]_e$ of the unit interval for every edge  $e \in E \cup \Lambda$. Now, identify every vertex or end $x$ with the copy of $0$ in $[0,1]_e $ for which $x$ is the tail of $e$ and  with the copy of $1$ in $[0,1]_f$ for which $x$ is the head of~$f$, for all $e,f \in E \cup \Lambda$.

For inner points $z_e\in [0,1]_e$ and $z_f\in [0,1]_f$ of edges  $e,f \in E \cup \Lambda$ we say that $z_e$ \emph{corresponds} to $z_f$ if both correspond to the same point of the unit interval. For $e \in E \cup \Lambda$ the point set obtained from $[0,1]_e$ in $\vert D \vert$ is an \emph{edge of $\vert D \vert$}.  The vertex or end that was identified with the copy of $0$ is the \emph{tail} of the edge of $|D|$ and the vertex or end that was identified with the copy of $1$ its \emph{head}.

We define the topological space \dmtop\ on $|D|$ by specifying the basic open sets.
For a vertex $v$ we take the collection of uniform stars of radius $\varepsilon$ around $v$ as basic open neighbourhoods. For inner points $z$ of edges $ [0,1]_e$ with $e \in E$ we keep the open balls around $z$ of radius $\varepsilon$ as basic open sets (considered as subsets of $[0,1]_e$). Here we make the convention that for edges $e$ (possibly limit edges) the open balls $B_\varepsilon(z)$ of radius $\varepsilon$ around points $z \in e$ is implicitly chosen small enough to guarantee $B_\varepsilon(z)\subseteq e$.

Neighbourhoods $\hat{C}_{\varepsilon}(X,\omega)$ of an end $\omega$ are of the following form: Given $X\in \cX(D)$ let $\hat{C}_{\varepsilon}(X,\omega)$ be the union of (see Figure~\ref{fig:NbhdEnds})
\begin{itemize}
    \item the point set of $C(X,\omega)$,
    \item the set of all ends and points of limit edges that live in $C(X,\omega)$ and
    \item  half-open partial edges $(\varepsilon,y]_e $ respectively   $[ y,\varepsilon)_e$ for every edge $e \in E \cup \Lambda $ for which $y$ is contained or lives in $C(X,\omega)$.
\end{itemize}

\begin{figure}[h]
\centering
\def\svgwidth{0.59\columnwidth}
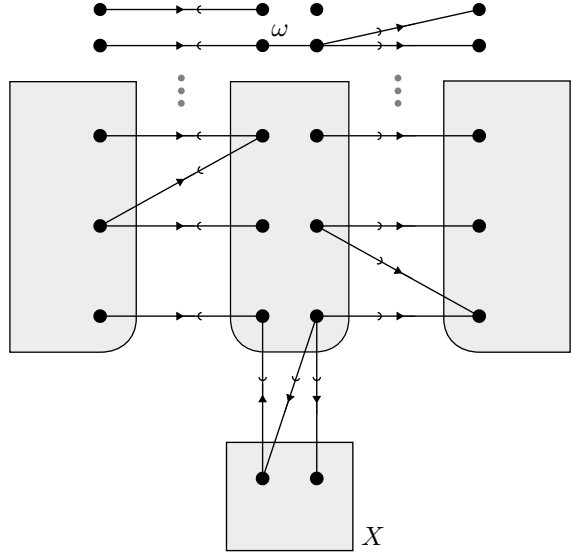
\caption{A basic open neighbourhood of the form $\hat{C}_{\varepsilon}(X,\omega)$.} 
\label{fig:NbhdEnds}
\end{figure} 

\noindent Neighbourhoods $\hat{E}_{\varepsilon,z}(X,\omega \eta)$ of an inner point $z$ of a limit edge $\omega\eta$ between ends are of the following form: Given $X\in \cX(D)$ that separates $\omega$ and $\eta$ let $\hat{E}_{\varepsilon,z}(X,\omega \eta)$ be the union of (see Figure~\ref{fig:NbhdEdges}) \begin{itemize}  
\item the open balls of radius $\varepsilon$ around points $z_e$ of edges $e\in E(X,\omega\eta)$ and with $z_e$ corresponding to $z$ and
\item the open balls of radius $\varepsilon$ around points $z_\lambda$ of limit edges $\lambda$ that live in the bundle $E(X,\omega\eta)$ and with $z_\lambda$ corresponding to $z$.
\end{itemize}

\newpage
\noindent Similarly, for an inner point $z$ of a limit edge $v\omega$ between a vertex $v$ and an  end $\omega$ we define the open neighbourhoods $\hat{E}_{\varepsilon,z}(X,v\omega)$ as follows. Given $X\in \cX(D)$ with $v \in X$ let $\hat{E}_{\varepsilon,z}(X,v \omega)$ be the union of  \begin{itemize}  
\item the open balls of radius $\varepsilon$ around points $z_e$ of edges $e\in E(X,v\omega)$ and with $z_e$ corresponding to  $z$ and
\item the open balls of radius $\varepsilon$ around points $z_\lambda$ of limit edges $\lambda$ that live in the bundle $E(X,v\omega)$ and with $z_\lambda$ corresponding to $z$.
\end{itemize}
Open sets $\hat{E}_{\varepsilon,z}(X,\omega v)$ for a limit edge $\omega v$ between an  end $\omega$ and a vertex $v$ are defined analogously.  
\begin{figure}[h]
\centering
\def\svgwidth{0.385\columnwidth}
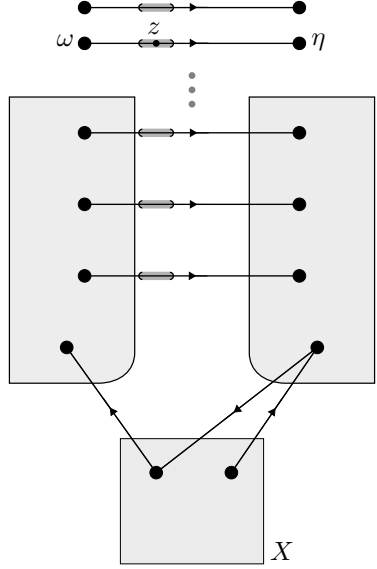
\caption{A basic open neighbourhood of the form $\hat{E}_{\varepsilon,z}(X,\omega \eta)$.} 
\label{fig:NbhdEdges}
\end{figure}

We remark that these basic open sets ensure that limit edges are homeomorphic to the unit interval and that the space $|D|$ is Hausdorff. We view a digraph $D$ as a subspace of $|D|$, namely the subspace that is formed by all the (equivalence classes of) vertices and inner points of edges of $D$. If there is no danger of confusion we will not distinguish between the digraph $D$ and the topological space $D$. Furthermore, we call the subspace $\Omega(D)$  of $\vert D \vert$ the \emph{end space of $D$}. The end space of an undirected graph $G$ coincides with the end space of the digraph obtained from $G$ by replacing every edge by its two orientations as separate directed edges.

One of the key definitions in the first paper of our series~\cite{EndsOfDigraphsI}, was that an end $\omega$ of $D$ is said to be in the closure of $\cU$, for a set of vertex sets $\cU$, if for all $X \in \cX(D)$ every $U \in \cU$ has a vertex in $C(X , \omega)$. Now that \dmtop\ is at hand this is tantamount to $\omega \in \overline{U}$ for every $U \in \cU$. We therefore obtain an extension of \cite[Lemma~4.1]{EndsOfDigraphsI}: 
\newpage

\begin{lemma} \label{lemma: connection closure; direction; necklace; EXTENDED top. closure} Let $D$ be any digraph, and let $\cU$ be a finite set of vertex sets of $D$. Then the following assertions are equivalent:
\begin{enumerate}
    \item $D$ has an end in the closure of $\cU$;
    \item $D$ has a vertex-direction in the closure of $\cU$;
    \item $D$ has a necklace attached to $\cU$;
    \item $D$ has an end in $ \bigcap \{\, \overline{U} \mid U \in \cU \,\} $.
\end{enumerate}
\end{lemma}

Recall that we call a digraph $D$ \emph{solid} if $D-X$ has finitely many strong components for every $X \in \cX(D)$. The main result of this section reads as follows:

\begin{mainresult}\label{thm: compact iff solid} The space $|D|$ is compact if and only if $D$ is solid. 
\end{mainresult}
\begin{proof}
We prove the forward implication by contraposition. If $D$ is not solid let~$X$ be a finite vertex set such that $D-X$ has infinitely many strong components. We obtain an open cover $\cO$ of $|D|$ that has no finite subcover as follows. Fix for every strong component $C$ of $D-X$ a vertex $u_c \in C$ and denote by $U$  the set of all the vertices $u_c$. It is straightforward to check that every point in $|D|\setminus U$ has a basic open neighbourhood that avoids $U$; this shows that $U$ is closed in $|D|$. Let $\cO$ consist of the uniform stars of radius $\frac{1}{2}$ around each $u_c$ and the open set $|D| \setminus U$. Then, $\cO$ is the desired open cover.

Now, let us prove the backward implication. For this, let $D$ be any solid digraph and let $\cO$ be an open cover of $|D|$. We may assume that $\cO$ consists of basic open sets. For every $X\in\cX(D)$ and every strong component $C$ of $D-X$, we let $\hat{C}$ be the union of the point set of $C$, the set of all the ends that live in $C$ and the  point set of all the limit edges that live in $C$. For a bundle $F$ of $D-X$, let $\hat{F}$ consist of the inner points of edges in $F$ and all the inner points of limit edges that live in $F$.  A strong component $C$ of $D-X$ is \emph{bad for $X$} if $\hat{C}$ is not covered by any cover set in~$\cO$. A bundle $F$ of $D-X$ is \emph{bad for $X$} if $\hat{F}$ is not covered by finitely many cover sets in $\cO$. A bad strong component for $X$ or a bad bundle for $X$  is a \emph{bad set for}~$X$. 

If there is no bad set for some $X\in \cX(D)$, we find a finite subcover as follows. For every strong component $C$ of $D-X$ fix a cover set from $\cO$ that covers $\hat{C}$. And for every bundle of $D-X$ fix finitely many cover sets from $\cO$ that cover $\hat{F}$. Note that our assumption that $D$ is solid ensures that there are only finitely many strong components and bundles of $D-X$. Therefore, we have fixed only finitely many cover sets in total. Combining these with a finite subcover of $D[X]$, which exists because $D[X]$ is a finite digraph, yields a finite subcover of $|D|$. Note that all the edges between vertices $x \in X$ and strong components of $D-X$ are covered, as they are bundles.

So let us assume for a contradiction that there is a bad set for every \hbox{$X\in \cX(D)$}. We will find a bad set for every \hbox{$X \in \cX(D)$} in a consistent way, i.e., for every two vertex sets $X,Y \in \cX(D)$ the bad set of $X$ contains that of $Y$ whenever $X \subseteq Y$. In other words the bad sets will give rise to a direction $f$ and we will then conclude that $f(X)$ is covered by finitely many sets in $\cO$ for some $X\in \cX(D)$, contradicting that $f(X)$ is~bad.

Given $X\in \cX(D)$, let $B_X$ be the union of all the sets $\hat{B}$ for which $B$ is bad for~$X$. It is straightforward to see that \{\,$B_X \colon X \in \cX $\,\} is a filter base on $|D|$ and we denote by $\cB$ some ultrafilter that extends it. On the one hand, $\cB$ contains for every $X\in \cX(D)$ at most one set $\hat{C}$ or $\hat{F}$ with $C$ a strong component of $D-X$ or $F$ a bundle of $D-X$, respectively, because intersections of filter sets are non-empty. On the other hand, there is at least one strong component $C$ or bundle $F$ of $D-X$ such that $\hat{C}$ or $\hat{F}$ is contained in $\cB$: Otherwise, $\cB$ contains $|D| \setminus \hat{C}$ and $|D| \setminus \hat{F}$ for every strong component of $D-X$ respectively every bundle $F$ of $D-X$. As $\cB$ does not contain the point set of $D[X]$ we have $|D| \setminus D[X] \in \cB$. But, the intersection of all the $|D| \setminus \hat{C}$ and $|D| \setminus \hat{F}$  with $|D| \setminus D[X]$ is empty. Consequently, $\cB$ contains for every $X \in \cX(D)$ exactly one set of the form $\hat{C}$ or $\hat{F}$ with $C$ a strong component of $D-X$ or $F$ a bundle of $D-X$, respectively. As intersections of filter sets are non-empty, these bundles and strong components form a direction $f$. Note, that for every $X\in \cX(D)$ the set $f(X)$ is bad for $X$ as it is the superset of $B_Y$ for some~$Y \in \cX(D)$.

In order to arrive at a contradiction we consider three cases. First, if $f$ is a vertex-direction, then by Theorem~\ref{thm: ends correspond bijectively to vertex directions}, we have that $f$ corresponds to an end $\omega$ which is covered by some cover set $O\in \cO$. As $O$ is a basic open set it is of the form $\hat{C}_\varepsilon(X,\omega)$ for some $X\in \cX(D)$. This contradicts that $f(X)$ is bad. 

Second, suppose that $f$ is an edge-direction and that $f$ corresponds  to a limit edge $\omega \eta$ between  ends in the sense of Theorem~\ref{thm: limit edges correspond bijectively to edge directions}. This limit edge  $\omega \eta$ is covered by a finite subset $\cO'\subseteq \cO$, as it is homeomorphic to the unit interval. Since each cover set $O\in \cO'$ is basic open it comes by its definition together with a finite vertex set $X_O\in \cX(D)$. Let $\cX':=\{\,X_O\mid O\in \cO' \,\}$ and let $X$ be large enough so that it contains $\bigcup \cX'$ and so that it separates $\omega$ and $\eta$. To get a contradiction, we show that $ \widehat{f(X)}$ is covered by $\cO'$. Consider a point $z \in  \widehat{f(X)}$ and let $z'$ be its corresponding point on $\omega \eta$. Then $z'$ is covered by some $O\in \cO'$. Since $X_O\subseteq X$ we have $f(X) \subseteq f(X')$ and therefore $O'$ also contains $z$. 

Finally, the case that $f$ is an edge-direction and $f$ corresponds to a limit edge between an  end and a vertex is analogue to the second case. \end{proof}

We complete this section by listing a few more properties that are equivalent to the assertion that $|D|$ is compact.

\begin{cor}\label{cor: listofeuiv.solid}
The following statements are equivalent for any digraph $D$: \begin{enumerate}
    \item $|D|$ is compact;
    \item every closed set of vertices is finite;
    \item $D$ has no $U$-rank for any infinite vertex set $U$\!;
    \item for every infinite set $U$ of vertices there is a necklace attached to $U$\!;
     \item $D$ is solid.
\end{enumerate}
\end{cor}
 
\begin{proof} 
(i)$\rightarrow$(ii): If $U\subseteq V(D)$ is closed and infinite, then any open cover that consists of $|D|-U$ and pairwise disjoint open neighbourhoods for the vertices in $U$ has no finite subcover.

(ii)$\rightarrow$(iii): Suppose that there is an infinite vertex set $U$ for which $D$ has a \hbox{$U$-rank $\alpha$.} We may choose $U$ so that $\alpha$ is minimal. Let $X \in \cX(D)$ witness that $D$ has $U$-rank $\alpha$. By the choice of $U$, all the strong components of $D-X$ contain only finitely many vertices of $U$. Hence, $U$ is closed in $|D|$, as every point in $|D|$ has an open neighbourhood that avoids $U$.

(iii)$\rightarrow$(iv) This is immediate by the necklace lemma.

(iv)$\rightarrow$(v) If $D$ is not solid, say $D-X$ has infinitely many strong components for $X \in \cX(D)$; then let $U$ be a vertex set that contains exactly one vertex of every strong component of $D-X$. Clearly, there is no necklace attached to $U$.

(v)$\rightarrow$(i) Theorem~\ref{thm: compact iff solid}.
\end{proof}

\section{The space $|D|$ as an inverse limit}\label{section: inverse limit}
\noindent In this section we show that the space $\vert D \vert$ for a solid digraph $D$ can be obtained as an inverse limit of finite contraction minors of $D$, Theorem~\ref{thm: |D| homeo inverse limit]}. We begin by defining an inverse system of finite digraphs for any digraph. Then, we  show that every digraph embeds in the inverse limit of its inverse system. This gives a compactification for arbitrary digraphs, Theorem~\ref{thm: compact. for any digraph}. 
 
Let us introduce an inverse system for a given digraph $D$. For this, we define a directed partially ordered set $(\cP,\le)$ as follows. We call a finite partition $P$ of  $V(D)$ \emph{admissible} if any two partition classes of $P$ can be separated in $D$ by a finite vertex set. We denote by $\mathcal{P}:=\cP(D)$ the set of all the admissible partitions of $D$. For any two partitions $P_{1}$ and $P_2$ of the vertex set of $D$ we write $P_{1} \leq P_2$ and say that   $P_2$ is  \emph{finer} than $P_{1}$ if every partition class of $P_2$ is a subset of a partition class~of $P_{1}$.

We claim that the set of admissible partitions is a directed partially ordered set. Indeed, the relation $\leq$ is easily seen to be a partial order on the set of  all the partitions of $V(D)$. In particular, it restricts to a partial order on the set of all the admissible partitions. To see that $\mathcal{P}$ is  directed, let $P, P^\prime \in \mathcal{P}$ be admissible partitions, and let $P^{\prime\prime}$ be the partition that consists of all the non-empty sets of the form $p \cap p^\prime$ with $p \in P$ and $p^\prime \in P^\prime$. Clearly,  $P^{\prime \prime}$ is finer than both $P$ and $P^\prime$. To see that $P''$ is admissible, let any two distinct  partition classes of $P^{\prime \prime}$ be given, say $p_1 \cap p_1^\prime$ and $p_2 \cap p_2^\prime$  with $p_1,p_2\in P$ and $p_1', p_2'\in P'$. As these partition classes are distinct, we have $p_1 \neq p_2$ or $p_1^\prime \neq p_2^\prime$, say $p_1 \neq p_2$. Since $P$ is admissible $D$ has a finite vertex set that separates $p_1$ and $p_2$, which in particular separates $p_1 \cap p_1^\prime$ and~$p_2 \cap p_2^\prime$.

Let us proceed by defining the topological spaces associated with the admissible partitions of $D$. Every admissible partition $P$ of $D$ gives rise to a finite (multi-) digraph $D/P$ by contracting each partition class and replacing all the edges between two partition classes by a single edge whenever there are infinitely many. Formally, declare $P$ to be the vertex set of $D/P$. Given distinct partition classes $p_1,p_2 \in P$ we define an edge $(e,p_1,p_2)$ of $D/P$ for every edge $e$ in $D$ from $p_1$ to $p_2$ if there are finitely many such edges. And if there are infinitely many edges from $p_1$ to $p_2$ we just define a single edge $(p_1p_2, p_1, p_2)$.  We call the latter type of edges  \emph{quotient edges}. Endowing $D/P$ with the $1$-complex topology turns it into a compact Hausdorff space, i.e., basic open sets are uniform $\varepsilon$ stars around vertices and open subintervals of edges. In other words, $D/P$ is defined as our topological space from the previous section (for finite $D$) with the only difference that multi-edges are taken into account. We will usually not distinguish between the finite \mbox{(multi-)} digraphs  $D/P$ and  the topological space $|D/P|$. 
Now, let us turn to the final ingredient of our inverse system for $D$: bonding maps. We define for every two distinct admissible partitions $P\le P'$ of $D$ a bonding map $f_{P' P} \colon D / P' \to D / P$ as follows. Vertices $p^\prime \in P^\prime$ of $ D / P^\prime$ get mapped to the unique vertex $p \in P$ of $D / P$ with $p^\prime \subseteq p$. Edges get mapped according to their endvertices: For edges $ (e',p_1',p_2')$ of $D / P'$ we consider two cases:  First, if  $p_1',p_2'\subseteq p$ for a partition class $p \in P$, then $(e',p_1',p_2')$ gets mapped to the vertex $p$ of $D / P$. Second, if $p_1' \subseteq p_1$ and $p_2' \subseteq p_2$ for two distinct partition classes $p_1,p_2\in P$, then there is at least one edge from $p_1$ to $p_2$ in $D/P$. If $ (e',p_1',p_2')$ is a quotient edge in $D / P'$, then also $(p_1p_2,p_1 , p_2)$ is a quotient edge in $D / P$ and we map $(e',p_1',p_2')$ to $(p_1p_2,p_1,p_2)$. If $ (e',p_1',p_2')$ is not a quotient edge in $D / P'$ and  $(p_1p_2,p_1,p_2)$ is a quotient edge in $D / P$, we map $(e',p_1',p_2')$ to $(p_1p_2,p_1,p_2)$. Finally, if $(e',p_1',p_2')$ is not a quotient edge and there is no quotient edge between $p_1$ and $p_2$, then $(e',p_1,p_2)$ is an edge in $D / P$ and we map $(e',p_1',p_2')$ to $(e',p_1,p_2)$.

It is straightforward to check that the $f_{P' P}$ are continuous and that we have  $f_{P^{\prime \prime}P} = f_{P^{\prime} P} \circ f_{P^{\prime \prime} P^{\prime} }$ for all admissible partitions $P\le P'\le P''$.
The bonding maps turn $\{ D/P,f_{P' P},\cP \}$ in an inverses system and we denote its inverse limit by~$\varprojlim (D/P )_{ P \in \mathcal{P} }$. Note that $\varprojlim (D/P )_{ P \in \mathcal{P} }$ is non-empty, as the collection of points that consists for every $P\in \mathcal{P}$ of the vertex of $D/P$ that contains a fixed vertex of $D$ is an element of $\varprojlim (D/P )_{ P \in \mathcal{P} }$.

Our next goal is to  find an embedding from $D$ to the inverse limit $\varprojlim (D/P )_{ P \in \mathcal{P} }$ witnessing that the inverse limit is a Hausdorff compactification of $D$. We obtain this embedding by defining continuous maps $\varphi_P \colon D \to D/P$, one for every admissible partition $P \in \cP$. Once the $\varphi_P$ are defined, the universal property of the inverse limit gives rise to the desired embedding.

So let us define $\varphi_P$ for a given admissible partition $P \in \cP$. For a vertex $v$ of $D$ let $\varphi_P(v)$ be the partition class of $P$ that contains $v$. For an inner point $z$ of an edge $vw$ of $D$, consider the partition classes that contain $v$ and $w$, respectively. If these coincide, map $z$ to the partition class that contains $v$ and $w$. Otherwise the partition classes that contain $v$ respectively $w$ differ and there is an edge $e$ from $\varphi_P(v)$ to $\varphi_P(w)$ in $D/P$, which is either a copy of $vw$ considered as an edge of $D/P$ or a quotient edge. Map $z$ to its corresponding point on $e$. It is straightforward to see that $\varphi_P$ is continuous for every $P \in \cP$ and that $f_{P' P} \circ \varphi_{P'} = \varphi_P$  for every admissible partitions $P\le P'$, i.e., the collection of maps $\{ \, \varphi_P \mid P \in \cP \, \}$ is compatible. Hence, by the universal property of the inverse limit, the $\varphi_P$ induce a map $ \Phi \colon D \to \varprojlim (D/ P )_{ P \in \mathcal{P} } $ with $\varphi_P = f_P \circ \Phi$, for every $P \in \mathcal{P}$.

\begin{theorem}\label{thm: compact. for any digraph}
For every digraph $D$ the space $\varprojlim (D/P )_{ P \in \mathcal{P} }$ is a Hausdorff compactification of $D$, in particular, the map $$ \Phi \colon D \to \varprojlim (D/P )_{ P \in \mathcal{P} } $$ is an embedding and its image is dense in $\varprojlim (D/P )_{ P \in \mathcal{P} }$.
\end{theorem}
\begin{proof}
We have to show that $\varprojlim (D/P )_{ P \in \mathcal{P} } $ is compact and Hausdorff, that the image of $\Phi$ is dense in $\varprojlim (D/P )_{ P \in \mathcal{P} } $ and that $\Phi$ is an embedding i.e., it is a homeomorphism onto its image. 
The inverse limit  $\varprojlim (D/P )_{ P \in \mathcal{P} } $ is compact and Hausdorff because all the topological spaces $D/P$ are compact and Hausdorff. As every $\varphi_P$ is surjective the image of $\Phi$ is dense in $\varprojlim (D/P )_{ P \in \mathcal{P} } $, by Lemma~\ref{lemma: image dense in inverse limit}. In order to show that $\Phi$ is a homeomorphism onto its image, note first that the collection of maps $\{ \, \varphi_P \mid P \in \cP \, \}$ is  eventually injective. Hence $\Phi$ is injective by the lifting lemma.

It remains to show that the inverse of $\Phi$ is continuous, for which we equivalently show that $\Phi$ is open onto its image, i.e.,  the image under $\Phi$ of open sets in $D$ is open in $\Phi(D)$. It suffices to show this on a base for the open sets in $D$. We prove that $\Phi$ is open for the base $\cB$ given by the open uniform stars around vertices and the open subintervals of edges. Our goal is to find for every $B \in \cB$ an open set $O$ such that $\Phi(B) = O \cap \Phi(D)$.  First consider the case where $B=B_\varepsilon(v)$ is an open ball of radius $\varepsilon$  around a vertex $v$. Then let $P$ be any admissible partition in which $\{ v \}$ is a singleton partition class. In $D / P$ we have that $\varphi_P (B) $ is an open ball of radius $\varepsilon$ around the vertex $\varphi_P(v)$. We claim that $O:= f_P^{-1} ( \varphi_P(B) ) $ is the desired open set. Clearly, $\Phi(B) \subseteq O \cap \Phi(D)$, we prove the converse inclusion. For this let $x \in O \cap \Phi(D)$ be given. Let $d \in D$ be the preimage of $x$ under $\Phi$. We have to show that $d \in B$. If $d \not\in B$, then $\varphi_P(d) \not\in \varphi_P(B)$, contradicting the fact that $x \in O$.

Second let $B$ be an open subinterval of an edge $e$ in $D$, say with end points $v$ and $w$. Then let $P$ be any admissible partition in which $\{ v \}$ and $\{ w \}$  are singleton partition classes. A similar argument as above shows that $O:= f_P^{-1} ( \varphi_P(B) ) $ is as desired.
\end{proof}

\noindent For example consider the directed ray $R$. Note first that every admissible partition of $R$ has exactly one infinite partition class. One can check that $\varprojlim (R/P )_{ P \in \mathcal{P} }$ is homeomorphic to the space where one adds a single point $\omega$ at infinity to $R$ and where a neighbourhood base of $\omega$ is given by the tails of $R$ together with $\omega$.

We now extend the maps $\varphi_P$ to maps   $ \hat{\varphi}_P \colon |D| \to D/P$. For this we define how $\hat{\varphi}_P$ behaves on ends and on inner points of limit edges; the values of $\hat{\varphi}_P$ on $D$ are then given by the values of $\varphi_P$ on $D$. For an  end $\omega$ of $D$ all the rays that represent $\omega$ have a tail in the same partition class $p$ of $P$. The reason for this is that any two partition classes of $P$ can be separated by a finite vertex set. Here we map $\omega$ to $p$. 

Now, consider an inner point $z$ of a limit edge~$\lambda$. Note that we have already defined the images of the two endpoints of~$\lambda$. If these images coincide, then map $z$ to the unique image of the endpoints of~$\lambda$. Otherwise, Proposition~\ref{lemma: char limit edge type I} or Proposition~\ref{lemma: char limit edge type II} gives rise to a quotient edge $\lambda'$  between the partition classes of the endpoints of~$\lambda$. In this case we map $z$ to the corresponding point on $\lambda'$. This completes the definition of $\hat{\varphi}_P$.

\begin{lemma}\label{lemma: phiHat are continuous}
The  map $\hat{\varphi}_P \colon |D| \to D/P  $ is continuous for every $P\in \cP$. 
\end{lemma}
\begin{proof}
In order to prove that $\hat{\varphi}_P$ is continuous, we show that the preimage of every open ball with radius $\varepsilon$ around  a vertex of $D/P$ is open in $|D|$  and that the preimage of every open subinterval of an edge in $D/P$ is open in $|D|$. As these open sets form a base of the topology of $D/P$, the map $\hat{\varphi}_P$ is continuous.

Consider an open ball $B_\varepsilon(p)$ of radius $\varepsilon$ around a vertex $p \in P$ in $D/P$. To see that $\hat{\varphi}_P^{-1}(B_\varepsilon(p))$ is open in $|D|$ we will define for every $y \in \hat{\varphi}_P^{-1}(p)$ an open set $O_y $ in $ |D|$ such that $\hat{\varphi}_P(O_y) \subseteq B_\varepsilon(p)$; in other words, the union of the open sets $O_{y}$ is included in $\hat{\varphi}_P^{-1}( B_\varepsilon(p))$. A closer look on the definition of the $O_y$  will show  that 
this latter inclusion is in fact an equality. 

So let $y \in |D|$ with $\hat{\varphi}_P(y)=p$ be given. To begin, if $y$ is a vertex of $D$ let $O_y$ be the open ball in $|D|$ of radius $\varepsilon$ around $y$. If $y$ is an inner point of an edge $e$ of $D$, then the whole edge $e$ is mapped to $p$ and we choose $O_y$ to be the interior of $e$. 
If $y$ is an end or an inner point of a limit edge, we fix a finite vertex set $X$ that separates $p$ from every other partition class in $P$. Note, that a strong component of $D-X$ is either contained in $p$ or is disjoint from $p$. If $y$ is an end, let $O_y$ be the basic open neighbourhood $\hat{C}_\varepsilon(X,y)$. Note that, by the choice of $X$, the strong component $C(X,y)$ is included in the partition class $p$. If $y$ is an inner point of a limit edge  $\lambda$, and $X$ separates the endpoints of this limit edge, then let $O_y=\hat{E}_{\varepsilon',y}(X, \lambda)$  with $\varepsilon'<\varepsilon$ small enough to fit into $\lambda$, i.e., such that  $B_{\varepsilon'}(y) \subseteq \lambda$ where the $B_{\varepsilon'}(y)$ is considered in the space $[0,1]_\lambda$;  otherwise let $C(X, \omega)$ be the strong component of $D-X$ that contains both endpoints of $\lambda$ and let $O_y= \hat{C}_\varepsilon(X,\omega)$.

Clearly, the union of the $O_y$ is included in $\hat{\varphi}_P^{-1}( B_\varepsilon(p))$. Moreover, for every $z  \neq  p$ in $B_\varepsilon(p)$ the set  $\hat{\varphi}_P^{-1}(z)$ is a set of inner points of  edges (possibly limit edges). Each such inner point is contained in an $\varepsilon$-neighbourhood of the endpoint $e$ that is mapped to $p$, for $e$ the edge that contains the inner point. Hence each of these inner points is contained in at least one of the open sets $O_y$.

Now, consider an open subinterval $B_\varepsilon(z)$ of radius $\varepsilon$ around $z$ for an inner point $z$ of an edge  $(e,p,p')$ of $D/P$. If $(e,p,p')$ is not a quotient edge of $D/P$, then $e$ is an edge of $D$ and the preimage of $B_\varepsilon(z)$ is an open subinterval of $e$ considered as an edge in $|D|$, namely around the point $\hat{\varphi}_P^{-1}(z)$ of radius $\varepsilon$. So suppose that $(e,p,p')$ is  a quotient edge.  We will find for every  point $y \in \hat{\varphi}^{-1}_P(z)$ an open neighbourhood $O_y$ of $y$ in $|D|$ with $O_y \subseteq \hat{\varphi}_P^{-1}( B_\varepsilon(z))$. 
A similar argument for every point in $B_\varepsilon(z)$ shows that $\hat{\varphi}_P^{-1} (B_\varepsilon(z) )$ is the union of open subsets in $\vert D \vert$.

Note that all the points in $\hat{\varphi}^{-1}_P(z)$ are inner points of edges in $|D|$ (possibly limit edges). Let $y \in \hat{\varphi}^{-1}_P(z)$ be given. First, if $y$ is an inner point of an edge of $D$, then let $O_y$ be the open subinterval of radius $\varepsilon$ around $y$. Second, suppose that $y$ is an inner point of a limit edge whose  end points are ends, say $\omega$ and $\eta$, and with $\hat{\varphi}_P(\omega)=p$ and $\hat{\varphi}_P(\eta)=p^\prime$. Fix finite vertex sets $X_p, X_{p^\prime}\subseteq V(D)$  that separate $p$ respectively $p^\prime$ from every other partition class in $P$. Note, that $X_p \cup X_{p^\prime} $ separates $\omega$ and~$\eta$. Now, every edge that is contained in or lives in $E(X_p \cup X_{p^\prime},\omega \eta)$ is mapped to $(e,p,p')$; thus the basic open neighbourhood  $\hat{E}_{\varepsilon,y}(X_p \cup X_{p^\prime},\omega\eta)$ is mapped to $B_\varepsilon(z)$. Finally, suppose that $y$ is an inner point of a limit edge $\lambda$ between a vertex $v$ and an  end $\omega$, say with $\hat{\varphi}_P(v)=p$ and $\hat{\varphi}_P(\omega)=p'$; the other case is analogue. Let $X_{p'}$ be a finite vertex set of $D$ that separates the partition class $p'$ from every other partition class in $P$. Then $\hat{E}_{\varepsilon,y}(X_{p'} \cup\{ v\}, \lambda )$ is mapped to $B_\varepsilon(z)$. 
\end{proof}

We are now ready to prove the main result of this section:

\begin{mainresult}\label{thm: |D| homeo inverse limit]}
Let $D$ be a solid digraph. The  map induced by  the $ \hat{\varphi}_P \colon |D| \to D/P $ $$\hat{\Phi}  \colon |D| \to \varprojlim (D/P )_{ P \in \mathcal{P} }$$ is a homeomorphism.
\end{mainresult}
\begin{proof}
It is straightforward to show that the $\hat{\varphi}_P$ are compatible. Let us show that the collection of maps $\{ \, \hat{\varphi}_P \mid P \in \cP \, \}$ is eventually injective, that is to say for  every two points $x,y \in |D|$ there is a $P\in \cP$ such that $\hat{\varphi}_{P}(x) \neq \hat{\varphi}_{P}(y)$. Such an admissible partition is easily defined if at least one of the points $x$ and $y$ lies in $D$. So suppose $x$ and $y$ are ends or inner points of limit edges. If $x$ and $y$ are both ends choose an $X \in \cX(D)$ that separates $x$ and~$y$. Then the admissible partition $P_X$ given by the strong components of $D-X$ and all the vertices in $X$ as singletons, is the desired partition. Similarly, if $x$ is an end and $y$ is an inner point of a limit edge of the form $\omega \eta$ for two ends of $D$, then choose an $X \in \cX(D)$ that separates all the ends in $\{ x, \omega,\eta \}$ simultaneously. Again the admissible partition given by the strong components of $D-X$ and all the vertices in $X$ as singletons is as desired. The other cases are analogue and we leave the details to the reader.

By the lifting lemma and Lemma \ref{lemma: phiHat are continuous} the $\hat{\varphi}_P$ induce a continuous injective map $\hat{\Phi} \colon |D| \to \varprojlim (D/P )_{ P \in \mathcal{P} } $.
By Lemma~\ref{lemma: image dense in inverse limit}  we have that the image of the map $\hat{\Phi}$ is dense in $\varprojlim (D/P )_{ P \in \mathcal{P} }$. Moreover, as $D$ is solid, we have that $|D|$ is compact by Theorem~\ref{thm: compact iff solid} so the image of $\hat{\Phi}$ is closed; hence it is all of $\varprojlim (D/P )_{ P \in \mathcal{P} }$.
The statement now follows from Lemma~\ref{lemma: compact into haussdorff gives hemeo}.
\end{proof}


In the  proof of Theorem~\ref{thm: |D| homeo inverse limit]} we used those admissible partitions that arise by deleting a finite vertex set from a solid digraph to ensure that the map $\Phi$ that is induced by the $\hat{\varphi}_P$ is injective. Next, we show that these admissible partitions capture the whole inverse system for a solid digraph. 

To make this formal, let $D$ be a solid digraph and $X\in\cX(D)$. We denote by $P_X$ the admissible partition where each vertex in $X$ is a singleton partition class and the other partition classes consist of the strong components of $D-X$. We claim that $\mathcal{P}_X:=\{\, P_X \mid X \in \cX(D)  \,\}$ is cofinal in the set of admissible partitions of $D$, that is for every admissible partition $P$ there is an $X$ such that $P \leq P_X$. Indeed, given $P\in\cP$ we have $P \leq P_X$ for any finite set $X\in \cX(D)$ that separates any two partition classes in $P$.

Now, $\{ D/P_X  ,  f_{P_{X} P_{X'}}, \mathcal{P}_X \}$ is an inverse system by itself and by Lemma~\ref{lemma: cofinal limit homeo} we have that $$ \varprojlim (D/P)_{P \in P}  \cong \varprojlim (D/ P_X)_{X \in \cX}.$$

\noindent If $D$ is countable one can simplify the directed system even further: Fix an enumeration $v_0,v_1,\ldots$ of the vertex set of $D$ and write $X_n$ for the set of the first $n$ vertices. Then the set of all the $P_{X_n}$ is  cofinal  in  $\mathcal{P}_X$ and therefore it is also cofinal in the set of all the admissible partitions of $D$.

\begin{cor}\label{cor: modD inverse limit of deletion minors}
Let $D$ be a countable solid digraph and let  $X_n$ consist of the first $n$ vertices of $D$ with regard to a any fixed  enumeration of $V(D)$. Then $|D|~\cong~\varprojlim(D/ P_{X_n})_{n \in \N } $.
\qed
\end{cor}

\section{Applications}\label{section: Applications} 
\noindent In this last section we prove  two statements about finite digraphs that naturally generalise to the space $|D|$, but do not generalise verbatim to infinite digraphs, Theorem~\ref{thm: eulerian characterisation} and Theorem~\ref{thm: char strong digraphs}. We begin this section by introducing all the definitions needed. We then provide an important tool that describes how (topological) paths in $\vert D \vert$ can pass through cuts in $D$, the directed jumping arc lemma. Finally, we prove our two main results of this section, Theorem~\ref{thm: eulerian characterisation} and Theorem~\ref{thm: char strong digraphs}. 

A continuous function $\alpha \colon [0,1] \to |D|$ is called  a \emph{local homeomorphism on the edges of $|D|$} if for every $x \in (0,1)$ that is mapped to an inner point of an edge $e \in E \cup \Lambda$, there is a neighbourhood $(a,b)$ of $x$ such that $\alpha$ restricts on $(a,b)$ to a homeomorphism to the interior of $e$, i.e., $\alpha \restriction(a,b) \cong \mathring{e}$. Note that by the continuity of $\alpha$ any such homeomorphism $\alpha \restriction(a,b)$ extends to a homeomorphism $\alpha \restriction[a,b] \cong e$. If in addition $\alpha$ \emph{respects the orientation of the edges} in $|D|$, that is if $[a,b] \subseteq [0,1]$ is mapped to an edge $e \in E \cup \Lambda$ we have that $x \leq y$ implies $\alpha(x) \leq \alpha(y)$ for all the $x,y \in [a,b]$, then we call $\alpha$ a \emph{directed topological path} in $|D|$. (Here $\alpha(x) \leq \alpha(y)$ refers to $\leq$ in $[0,1]_e$.)

We think of directed topological  paths in $|D|$ as generalised directed walks in $D$. Here the edges of a directed walk in $D$ are directed along the walk. Indeed, every directed walk in $D$ defines via a suitable parametrisation a directed topological  path in $|D|$. If the image of $\alpha$ contains a vertex or an end $x$ we simply say that $\alpha$ contains~$x$. We say that a directed topological path $\alpha$ \emph{traverses an edge} $e \in E \cup \Lambda$ of $|D|$ if $\alpha$ restricts on a subinterval  of $[0,1]$ to a homeomorphism on $e$. 
The points  $\alpha(0)$ and $\alpha(1)$ are called the \emph{endpoints} of $\alpha$ and we say that $\alpha$ \emph{connects} $\alpha(0)$ to~$\alpha(1)$.  A directed topological  path whose endpoints coincide is \emph{closed}.

Next, let us gain some understanding of how directed topological  paths  in $|D|$ can pass through cuts of $D$, see also the jumping arc lemma \cite[Lemma~8.6.3]{DiestelBook5}.

\begin{lemma}[Directed Jumping Arc Lemma]\label{lemma: jump lemma}  
Let $D$ be any digraph and let $\{V_1,V_2  \}$ be any bipartition of $V(D)$.
\end{lemma}
\begin{enumerate}
    \item If $\overline{V_1} \cap \overline{V_2} = \emptyset$, then  every directed topological  path in $|D|$ from $V_1$ to $V_2$ traverses an edge of $|D|$ with tail in $\overline{V_1}$ and head in $\overline{V_2}$.  
   \item If $\overline{V_1} \cap \overline{V_2} \neq \emptyset$, there will be a directed topological  path in $|D|$ from $V_1$ to $V_2$ that traverses none of the edges between  $\overline{V_1}$ and $\overline{V_2}$ if both $D[V_1]$ and $D[V_2]$ are solid. 
\end{enumerate}
\begin{proof}
(i)  Suppose that $\overline{V_1}\cap \overline{V_2}$ is empty. Then every  end of $D$ is either contained in $\overline{V_1}$ or $\overline{V_2}$. First, we show that every edge of $|D|$ that has both of its endpoints in $\overline{D[V_i]}$ is contained in $\overline{D[V_i]}$, for $i=1,2$. For edges of $D$ this is trivial. So consider a limit edge $\lambda$ with both endpoints in $\overline{D[V_i]}$. All but finitely many vertices of a subdigraph obtained by Proposition~\ref{lemma: char limit edge type I} or Proposition~\ref{lemma: char limit edge type II} applied to $\lambda$ are contained in $D[V_i]$, otherwise this gives an end in $\overline{V_1}\cap \overline{{V_2}}$. Consequently, for every inner point $z \in \lambda$ there is a sequence of inner points of edges in $\overline{D[V_i]}$ that converge to $z$, giving $z \in \overline{D[V_i]}$.

From this first observation, we now know that $|D| \setminus (\overline{D[V_1]}  \cup \overline{D[V_2]})$ consists only of inner points of edges (possibly limit edges) between $\overline{V_1}$ and $\overline{V_2}$. Now, consider a directed topological  path $\alpha$ that connects a point in $V_1$ to a point in~$V_2$.
As $[0,1]$ is connected and $\alpha$ is continuous there is a point $x\in [0,1]$ with $\alpha(x)\in |D|\setminus(\overline{D[V_1]} \cup \overline{D[V_2]})$. Hence the preimage of $|D|\setminus(\overline{D[V_1]} \cup \overline{D[V_2]})$ is non-empty and a union of pairwise disjoint intervals $(a,b)$ each of which is mapped homeomorphically to an open edge between $\overline{V_1}$ and $\overline{V_2}$. The usual relation $\le$ on the reals defines a linear order on these intervals. Among these intervals, choose $(a,b)$ minimal. That this is possible can be seen as follows: If not, we find a strictly decreasing sequence $(a_0,b_0)  \geq (a_1,b_1) \geq \dots$ of intervals with $\alpha \restriction[a_i,b_i] \cong e_i$ for some edges $e_i$ between $\overline{V_1}$ and $\overline{V_2}$. Then $(a_i)_{i\in\N}$ and $(b_i)_{i\in\N}$ converge to some $c\in[0,1]$ and using that $\alpha$ is continuous, we get $\alpha(c)\in \overline{V_1}\cap \overline{V_2}$, a contradiction.

We claim that the image of $(a,b)$ under $\alpha$ is an edge from $\overline{V_1}$ to $\overline{V_2}$. To see this, it suffices to show that $\alpha(a) \in \overline{V_1}$. So suppose for a contradiction that $\alpha(a) \in \overline{V_2}$. Then $\alpha \restriction[0,a]$ gives a directed topological  path from $\overline{V_1}$ to $\overline{V_2}$. By a similar argument as above there is a point in $[0,a]$ mapped to an edge between $\overline{V_1}$ and $\overline{V_2}$, contradicting the choice of $(a,b)$.

(ii) First note that no inner point of a limit edge is a limit point of a set of vertices. Hence $D$ has at least one  end that is contained in both the closure of $V_1$ and $V_2$. By Lemma~\ref{lemma: connection closure; direction; necklace; EXTENDED top. closure} we find a necklace $N\subseteq D$ attached to $\{V_1,V_2 \}$. Let $\omega$ be the end that is represented by  $N$.  Apply Corollary~\ref{cor: listofeuiv.solid} to the solid digraph $D[V_1]$ and the infinite set $U_1:= V_1\cap V( N)$ in order to obtain a necklace $N_1$ attached to~$U_1$. Let $U_2$ consist of all the vertices in $V_2$ that are contained in those beads of $N$ that intersect a bead of $N_1$. Apply Corollary~\ref{cor: listofeuiv.solid} to the solid digraph $D[V_2]$ and the infinite set $U_2$ in order to obtain a necklace $N_2$ attached to~$U_2$. Note that both necklaces $N_1$ and $N_2$ represent $\omega$. A ray in $N_1$ together with a reverse ray in $N_2$  defines a directed topological  path that is as desired.
\end{proof}

Now, let us turn to our applications. A finite digraph is called \emph{Eulerian} if there is a closed directed walk that contains every edge exactly once.  A \emph{cut} of a digraph $D$ is an ordered pair $(V_1,V_2)$ of non-empty sets $V_1,V_2 \subseteq V(D)$ such that $V_1 \cup V_ 2 = V(D)$ and $V_1 \cap V_2 = \emptyset$. The sets $V_1$ and $V_2$ are the \emph{sides} of the cut, and its \emph{size} is the cardinality of the set of edges from $V_1$ to $V_2$.  We call a cut $(V_1,V_2)$ \emph{balanced} if its size equals that of $(V_2,V_1)$.  An \emph{unbalanced cut} is a cut that is not balanced. It is  well known that a finite digraph (with a connected underlying graph) is Eulerian if and only if all of its cuts are balanced. 

A closed directed topological path $\alpha$ that traverses every edge of $|D|$ exactly once is called \emph{Euler tour}, i.e., for every edge $e$ of $|D|$ there is exactly one subinterval of $[0,1]$ that is mapped homeomorphiccally to $e$ via $\alpha$. If $|D|$ has an Euler tour we call $|D|$ \emph{Eulerian}. There are two obstructions for digraph $D$ to be Eulerian: one is a vertex of infinite degree and the other one is an unbalanced cut. A digraph is \emph{locally finite} if all its vertices have finite in- and out-degree. Theorem~\ref{thm: eulerian characterisation} states that there are no further obstructions. We need one more lemma for its proof:

\begin{lemma}\label{lemma: loc. fin. and balanced gives solid}
Let $D$ be a digraph with a connected underlying graph. If $D$ is locally finite and every finite cut of $D$ is balanced, then $D$ is solid.
\end{lemma}
\begin{proof}
Suppose for a contradiction that $D$ is not solid and fix a finite vertex set $X \subseteq V(D)$ such that $D-X$ has infinitely many strong components. Our goal is to find a finite unbalanced cut of $D$. We may view the strong components of $D-X$ partially ordered by $C_1 \leq C_2$ if there is a path in $D-X$ from $C_1$ to $C_2$.  We first note that any strong component $C$ of $D-X$ receives and sends out only finitely many edges in $D-X$. Indeed, if $C$ sends out infinitely many edges, then $(V_1,V_2)$ is a finite unbalanced cut, where $V_1$ is the union of all the strong components  strictly greater than $C$ and where $V_2:= V(D)\setminus V_1$. A similar argument shows that $C$ receives only finitely many edges. Now, the (multi-)digraph $D'$ obtained from $D$ by contracting all the strong components of $D-X$ is locally finite. Note that also every finite cut of $D'$ is balanced.

Now, $D'$ is also strongly connected. Indeed, if there is a vertex $v\in V(D')$ that cannot reach all the other vertices, then  $(V_1,V_2)$ is a finite unbalanced cut of $D'$, where $V_1$ is the set of vertices in $V(D')$ that can be reached from $v$ and where $V_2:= V(D')\setminus V_1$ (here we use that the graph underlying $D$ is connected). Hence we may apply the directed star-comb lemma in $D'$ to~$V(D')$. As $D'$ is locally finite, the return is a comb and a reverse comb sharing their attachment sets; we may assume that both avoid~$X$. Let $R$ be the spine of the comb and $R'$ the spine of the reverse comb. 
Let $V_1$ be the set of all the vertices in $D'-X$ that can be reached from $R'$ in $D'-X$ and $V_2:= V(D') \setminus V_1$.

As $D'-X$ is acyclic we have that the vertex set of $R$ is included in $V_2$. But then $(V_1, V_2)$ is a finite unbalanced cut, which in turn gives rise to a finite unbalanced cut of $D$. 
\end{proof}

\begin{mainresult}\label{thm: eulerian characterisation}
For a digraph $D$ with a connected underlying graph the following assertions are~equivalent:
\begin{enumerate}
    \item $|D|$ is Eulerian;
    \item $D$ is locally finite and every finite cut of $D$ is balanced.
\end{enumerate}
\end{mainresult} 
\begin{proof}
For the forward implication (i)$\to$(ii) suppose that $D$ has an Euler tour $\alpha$. Using the directed jumping arc lemma it is straightforward to show that $D$ has only balanced cuts.  Let us show that $D$ needs to be locally finite for $\alpha$ to be continuous. Suppose for a contradiction there is a $v \in V(D)$ with infinitely many edges $e_0,e_1,\ldots$ with head $v$; the case where $v$ is the tail of infinitely many edges is analogue.  Let $(a_i,b_i) \subseteq [0,1]$ the subinterval that is mapped homeomorphic by $\alpha$ to $e_i$. As the unit interval is compact, the sequence of the $a_i$ has a convergent subsequence $(a_{i_n})_{n \in \N}$ and we write $x$ for the limit point of this subsequence. Now, the subsequence  $(b_{i_n})_{n \in \N}$ of the $b_i$ forms a convergent subsequence, too, with limit point $x$. As $\alpha(b_{i_n}) = v$ for all the $n \in \N$ we have $\alpha(x)=v$, by the continuity of $\alpha$; but $\alpha(a_{i_n})$ is a sequence of neighbours of $v$ which does not converge in $\vert D \vert$ to $v$, a contradiction.

For the backward implication (ii)$\rightarrow$(i) let us first show that $|D|$ contains no limit edges. As $D$ is locally finite there is no limit edge between a vertex and an end. So suppose for a contradiction there is a limit edge $\omega \eta$ between two  ends of $D$. Fix a finite vertex set $X$ that separates $\omega$ and $\eta$. We may view the strong components of $D - X$ partially ordered by $C_1 \leq C_2$ if there is a path in $D-X$ from $C_1$ to~$C_2$. Let $V_1$ consist of all the vertices in strong components of $D-X$ that are strictly smaller than $C(X,\eta)$ and let $V_2:= V(D)\setminus V_1$. Then $(V_1,V_2)$ is an unbalanced cut: On the one hand there are infinitely many edges from $V_1$ to $V_2$  because there are infinitely many from $C(X,\omega)$ to $C(X,\eta)$. On the other hand, there are only finitely many edges from $V_2$ to $V_1$   by our assumption that $D$ is locally finite.

Let us now find an Euler tour for $|D|$. By Lemma~\ref{lemma: loc. fin. and balanced gives solid} the digraph $D$ is solid. As it is locally finite and its underlying graph is connected $V(D)$ is countable. Choose an enumeration of $V(D)$ and let $X_n$ denote the set of the first $n$ vertices. Then $D/P_{X_n}$ contains no quotient edge and every cut of $D/P_{X_n}$ is balanced. As the statement of Theorem~\ref{thm: eulerian characterisation} holds for finite digraphs, we have that $D/P_{X_n}$ is Eulerian. Moreover, as $D/P_{X_n}$ is a finite digraph there are only finitely many (combinatorial) Euler tours of $D/P_{X_n}$. By K\H{o}nig's infinity lemma there is a consistent choice of one Euler tour for every $D/P_{X_n}$. 
Now, take a parametrisation $\alpha_n \colon [0,1] \to  D / P_{X_n} $ of the Euler tour chosen for $D / P_{X_n}$ such that the $\alpha_n$ are compatible. 
Using Theorem~\ref{thm: |D| homeo inverse limit]} it is straightforward to check that the universal property of the inverse limit gives an Euler tour for~$|D|$. 
\end{proof}

It is  well know that a finite digraph is strongly connected if and only if it has a directed closed walk that contains all its vertices. Clearly, the statement does not generalise verbatim to infinite digraphs nor does a spanning directed (double)--ray ensure the digraph to be strongly connected. Moreover, the statement does not hold if one adds the ends of the underling undirected graph:

\begin{figure}[h]
\centering
\def\svgwidth{0.8\columnwidth}
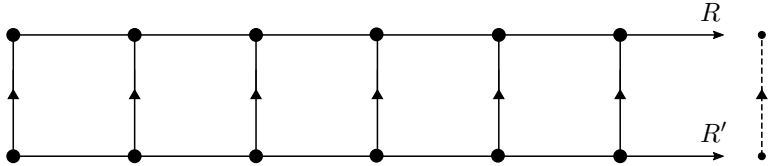
\caption{A solid digraph (every undirected edge in the figure stands for two directed edges in opposite directions) that is not strongly connected. Adding the one end of the underling undirected graph makes it possible to find a closed directed  topological path that contains all the vertices.}
\label{fig:ladder}
\end{figure} 

Adding the ends and limit edges of the digraph turns out to be the right setting for the statement to generalise:

\begin{mainresult}\label{thm: char strong digraphs}
For a countable solid digraph $D$ the following assertions are equivalent:
\begin{enumerate}
    \item $D$ is strongly connected;
    \item there is a closed topological path in $|D|$ that contains all the vertices of $D$.
\end{enumerate}
\end{mainresult}
\begin{proof}
For the forward implication (i)$\to$(ii) fix an  enumeration $v_1,v_2,\ldots$ of $V(D)$ and denote by $X_n$ the set of the first $n$ vertices. We will recursively define a sequence of walks $W_1, W_2 , \ldots$ such that $W_n$ is a directed closed walk of $D / P_{X_n}$ that contains all the vertices of $D / P_{X_n}$ and such that the projection of $W_n$ to $D / P_{X_{n-1}}$ is exactly~$W_{n-1}$. 

Once the $W_n$ are defined, it is not hard to find parametrisations $\alpha_n$ of each $W_n$ such that $f_{X_n X_{n-1}} \circ \alpha_n = \alpha_{n-1}$. Then the universal property of the inverse limit together with Corollary~\ref{cor: modD inverse limit of deletion minors} and Theorem~\ref{thm: |D| homeo inverse limit]} gives the desired closed directed topological path in~$|D|$. 

To begin, let $W_1$ be an arbitrary closed walk in $D/ P_{X_1}$ that contains all its vertices. Now suppose that $n>1$ and that $W_{n-1}$ has already been defined. Let $C$ be the strong component of $D-X_{n-1}$ that contains $v_n$. Note that the strong components of $D-X_n$ are exactly the strong components  of $D-X_{n-1}$ that are distinct from $C$ together with all the strong components of $C - v_n$. As $C$ is strongly connected the digraph $C /P_{ \{v_n\}}$ is strongly connected, as well. We now extend $W_{n-1}$ to $W_n$ by plugging in a directed walk that contains all the vertices of $C/P_{\{v_n\}}$ each time $W_{n-1}$ meets $C$. Formally, we fix for every edge $e_i$ of $W_{n-1}$ with one of its endvertices in $C$ an edge $f_i$ in $D / P_{X_n}$ that is mapped to $e_i$ by $f_{X_n X_{n-1}}$. For every occurrence of $C$ in $W_{n-1}$ there are consecutive edges $e_i$ and $e_{i+1}$ in $W_{n-1}$ such that $C$ is the head of $e_i$ and the tail of $e_{i+1}$. Now, fix a directed walk $Q_i$ in $C/P_{\{v_n\}}$  from the head of $f_i$ to the tail of $f_{i+1}$ that contains all the vertices of $C / P_{\{ v_n\}}$. We define $W_n$ by replacing any such consecutive edges $e_i$ and $e_{i+1}$ in $W_{n-1}$ by $f_i Q_i f_{i+1}$.

We prove the implication (ii)$\to$(i) via contraposition. Suppose that $D$ is not strongly connected. Then there are vertices $v,w\in V(D)$ so that there is no path from $v$ to $w$. Let $V_1$ consist of all the vertices that can be reached from $v$ and let $V_2:= V(D)\setminus V_1$. As $w \not\in V_1$, we have $V_2 \neq \emptyset$. Moreover, the edges between $V_1$ and $V_2$ form a cut with no edge from $V_1$ to $V_2$ (in particular no limit edge). Hence the intersection $\overline{V_1} \cap \overline{V_2}$ is empty. By the directed jumping arc lemma there is no directed topological  path in $|D|$ from $v$ to $w$. We conclude that there is no closed directed topological path in $|D|$ that contains all the vertices of~$D$.
\end{proof}

\bibliographystyle{amsplain}
\bibliography{bibliography.bib}
\end{document}